\documentclass[12pt]{amsart}

\usepackage{graphicx}
\usepackage{amssymb}
\usepackage{xcolor}
\usepackage{bbm}
\usepackage{tikz}

\newtheorem{theorem}{Theorem}[section]
\newtheorem{proposition}[theorem]{Proposition}
\newtheorem{corollary}[theorem]{Corollary}
\newtheorem{lemma}[theorem]{Lemma}
\theoremstyle{definition}
\newtheorem{definition}[theorem]{Definition}
\newtheorem{notation}[theorem]{Notation}
\newtheorem{example}[theorem]{Example}
\newtheorem{remark}[theorem]{Remark}

\newcommand{\PP}{\mathbb{P}}
\newcommand{\V}{Var}

\begin{document}\thispagestyle{empty}






\title[random monomial ideal associated to graphs]
{Probabilities of random monomial ideals associated to large graphs} 

\author[D. Munoz George]{Daniel Munoz George}
\address{
Department of Mathematics and Statistics\\
University of Hong Kong\\ 
Pok Fu Lam, Hong Kong
}
\email{dmunozgeorge@gmail.com}

\author[H. Mu\~noz-George]{Humberto Mu\~noz-George}
\address{
Department of Mathematics\\
Centro de Investigaci\'on y de Estudios
Avanzados del
IPN\\
Apartado Postal
14--740 \\
Ciudad de M\'exico, M\'exico, CP 07000.
}
\email{humbertomgeorge@gmail.com}

\author[K. Mu\~noz George]{Kevin Mu\~noz George}
\address{
Faculty of Sciences\\
Universidad Nacional Autónoma de México\\
Ciudad de M\'exico, M\'exico, CP 04510.
}
\email{kevinmunozgeorge@gmail.com}

\footnotetext[1]{The first author was supported by Hong Kong RGC Grant 16303922, NSFC12222121 and NSFC12271475}
\footnotetext[2]{The second author was supported by a scholarship from CONACYT, M\'exico.}

\maketitle 

\begin{abstract}
Inspired by the Erdős–Rényi model, we propose a new model for freesquare random monomial ideals generated by edges and covers of a graph. This permit us to investigate the conditions of normality for which we obtain asymptotic results. We also elaborate on asymptotic results for other invariants such as the Krull dimension (for which we obtain threshold function), the regularity and the $v$-number.    
\end{abstract}


\section{Introduction}\label{Section: introduction}

Sometimes in mathematics, a way to understand a complex object is to use randomness as a tool to study the average behavior of the object. There is already an extensive literature in probability which has proved to be effective for the study of complicated objects. To mention a few examples, the probabilistic method has been applied to large structures \cite{AS,CWY} and the study of large random matrices to understand its eigenvalue distribution \cite{V,Sh11,LP1,LSX,KKP,CES23,GB24}. In algebra, probability has also been used to study objects such as the study of random walks in random groups \cite{G,CCMP,HMZ} and the study of random permutations \cite{LWW,JHB}. In combinatorics, probability has been used to study large random graphs \cite{G59,ER}. Finally, there is also smaller body of literature on the study of random monomial ideals \cite{ranmon,SWY,BPEY,PSW}. 

Monomial ideals constitute one of the most widely studied classes of ideals in commutative algebra, not only because of their intrinsic algebraic richness, but also due to their deep and fruitful connections with combinatorics, geometry, and optimization. The study of algebraic invariants through combinatorial data has become a central theme. In particular, classical references such as the books by Herzog–Hibi \cite{Herzog} and R. H. Villarreal \cite{monalg-rev} highlight how invariants like depth, dimension, regularity, and normality can often be interpreted in purely combinatorial or polyhedral terms when dealing with monomial ideals.

Within this framework, there are two important types of monomial ideals associated to a given graph (or, more generally, to a clutter): \textit{edge ideals} and \textit{cover ideals}. These arise as two fundamental constructions linking squarefree monomial ideals with graph theory. Given a simple graph, its edge ideal is generated by all squarefree monomials of degree two corresponding to the edges of the graph, while the cover ideal is generated by monomials corresponding to its minimal vertex covers. These ideals form two dual and deeply interconnected classes that have been studied extensively, as they explicitly reflect the combinatorial structure of the underlying graph. This perspective has produced a rich body of literature relating algebraic properties such as normality, unmixedness, regularity, and the behavior of symbolic and ordinary powers to classical graph-theoretic invariants. There is also a growing body of literature on the general theory of monomial ideals and, in particular, on edge and cover ideals \cite{Eisen, Herzog, monalg-rev}. 

Inspired in the work of De Loera et al. \cite{ranmon} and the Erdős–Rényi model, we propose a new model for random monomial ideals generated by the edges and the covers of a graph. In the former case, we obtain ideals generated by monomials of degree two while in the latter case the generators can have arbitrary degree. In particular, in the former setting we are able to study the normality of the ideal. It is important to note that this has not been done previously. This is likely due to the absence of general normality criteria for arbitrary monomial ideals; however, in the setting of squarefree monomial ideals generated by the edges of a graph, normality criteria has been extensively studied \cite{icdual, graphs}. In addition to normality normality, we investigate the Krull dimension of the ideal. In our model, the Krull dimension coincides with the independence number of the graph. The independence number of a graph is the size of the largest independent set of vertices, that is, vertices that are disconnected. Thanks to the previous fact, our algebraic problem becomes a problem in random graphs. We determine threshold function for the property of an Erdős–Rényi graph having the empty graph with $t$ vertices as an induced subgraph. This allow us to find a threshold function for the property of the Krull dimension being larger or equal than $t$. Here $t$ is an arbitrary positive integer. Finally we use our main theorems to deduce asymptotic results on the regularity and the $v$-number of the ideal. 

\section{Model and results}\label{Section: model and results}

Let us first introduce our model. Let $K$ be a field and let $S=K[x_1,\dots,x_n]$ be the polynomial ring in $n$ variables. We consider the random graph $\mathcal{G}=(V,E)$ whose vertices are indexed by the variables $V=\{x_1,\dots,x_n\}$ and each edge $\{x_i,x_j\}$ is chosen with probability $p$ for $i\neq j$. Each realization $G=(V,E)$ of the random graph $\mathcal{G}$ induces two ideals; $I(G)$ and $I_c(G)$, called the edge ideal and the ideal of covers respectively which are defined as follows.

\begin{enumerate}
    \item $I(G)$ is a squarefree monomial ideal generated by the monomials $x_ix_j$ for each $\{x_i,x_j\}\in E$.
    \item $I_c(G)$ is a squarefree monomial ideal generated by the monomials $x_{i_1}\cdots x_{i_r}$ where $\{x_{i_1},\dots,x_{i_r}\}$ is a minimal vertex cover of the graph.
\end{enumerate}

A more detailed definition of $I(G)$ and $I_c(G)$ with examples is left to definition \ref{Definition: The edge and ideal of covers}.

With the previous definitions the ideals $I(\mathcal{G})$ and $I_c(\mathcal{G})$ are random ideals of $S$. As our model is inspired by the Erdős–Rényi model we say the ideals have Erdős–Rényi distribution which we denote by $I(\mathcal{G})\sim IER(n,p)$ and $I_c(\mathcal{G})\sim IER_c(n,p)$.

Note that if $B\subset S$ is a subset of monomial ideals of degree two, say $B=\{x_{i_k}x_{i_l}: 1\leq i_k,i_l\leq n\}$ then

$$\PP(I(\mathcal{G})=(B))=p^{|B|}(1-p)^{\binom{n}{2}-|B|},$$
where $(B)$ denotes the ideal generated by $B$.

\subsection{results} Our results describe how the probability of a property varies depending on the parameters $p$ and $q=:1-p$. 

\textit{(A) Normality.} Let $R$ be a ring and let $I$ be an ideal of $R$, an element $z \in R$ is \textit{integral} over $I$ if $z$ satisfies an equation
\begin{align*}
z^l+a_1z^{l-1}+...+a_{l-1}z+a_l=0, && a_i \in I^i.
\end{align*}
The \textit{integral closure}\index{integral closure} of $I$ is the set of all elements $z \in R$ which are integral over $I$. This set will be denoted by $\overline{I}$

\begin{definition}
If $I=\overline{I}$, $I$ is said to be \textit{integrally closed}\index{integrally closed} or \textit{complete}\index{complete}.
\end{definition}

\begin{definition}
If $I^n$ is integrally closed for all $n\geq 1$, $I$ is said to be \textit{normal}\index{normal}.
\end{definition}

Our first result provides asymptotic probabilities for the properties $I(\mathcal{G})$ being normal, $I_c(\mathcal{G})$ being normal and $I_c(\mathcal{G})\text{ not being normal}\cap \beta_0(G)\leq 2$. Here $\beta_0(G)$ denotes the independence number of the graph $G$ as defined in Section \ref{Section: Monomial ideals}. We prove under suitable conditions on the parameters $p$ and $q$ the events (or their complements) hold asymptotically almost surely. 

\begin{theorem}\label{Theorem: Main theorem normality 1}
Let $I(\mathcal{G})$ be a random monomial ideal with Erdős–Rényi distrution, i.e., $I(\mathcal{G})\sim IER(n,p)$. Then
\begin{equation}
\lim_{n\rightarrow \infty}\PP(I(\mathcal{G})\text{ is normal})= \left\{ \begin{array}{lcc} 1 & \text{if} & p=o(\frac{1}{n}) \\ \\ 0 & \text{if} & pq^{3/2}=\omega(\frac{1}{n}) \end{array} \right.
\end{equation}
\end{theorem}

\begin{remark}
Observe that $pq^{3/2}\leq p$. If $p<<\frac{1}{n}$ then $I(\mathcal{G})$ is normal asymptotically almost surely. If $\frac{1}{n}<<pq^{3/2}$ then $I(\mathcal{G})$ is not normal asymptotically almost surely. It remains to understand the behavior of the event
$$\{I(\mathcal{G})\text{ is normal}\}$$
for $pq^{3/2} << \frac{1}{n} << p$.  \end{remark}

\begin{theorem}\label{Theorem: Main theorem normality 3}
Let $I_c(\mathcal{G})$ be a random monomial ideal with Erdős–Rényi distribution, i.e., $I_c(\mathcal{G})\sim IER_c(n,p)$. If $p=o(\frac{1}{n})$ then
\begin{equation}
\lim_{n\rightarrow \infty}\PP(I_c(\mathcal{G})\text{ is normal})=1
\end{equation}
\end{theorem}

\begin{theorem}\label{Theorem: Main theorem normality 2}
Let $I_c(\mathcal{G})$ be a random monomial ideal with Erdős–Rényi distribution, i.e., $I_c(\mathcal{G})\sim IER_c(n,p)$. If $q=o(\frac{1}{n})$ then
\begin{equation}
\lim_{n\rightarrow \infty}\PP(I_c(\mathcal{G})\text{ is not normal} \cap \beta_0(\mathcal{G})\leq 2)=0.
\end{equation}
\end{theorem}

\begin{remark}
The reason we intersect with the event $\beta_0(\mathcal{G})\leq 2$ in Theorem \ref{Theorem: Main theorem normality 2} is that for ideals of covers if the independence number of the graph is less than or equal to 2 then the following sentences are equivalent;
\begin{enumerate}
    \item The ideal is not normal.
    \item The complement of the graph admits a Hochster configuration.
\end{enumerate}
With the help of this equivalence we can proceed similarly to the proof of Theorem \ref{Theorem: Main theorem normality 1}. We leave the details to Section \ref{Section: Proof of main results}.
\end{remark}

\textit{(B) Krull dimension.} Let us recall that the Krull dimension of a monomial ideal $I$ is define as $\text{dim}(S/I)$. In the setting of monomial ideals generated by edges of a graph, $G$, the Krull dimension coincides with the independence number of the graph $\beta_0(G)$. Roughly speaking, the independence number is the size of the largest independent set of the graph. For our second result, we obtain a threshold function for the independence number of the graph being larger or equal than $t$ for every fixed integer $t$. As a corollary we prove that under suitable conditions on the parameter $q$ the Krull dimension equals $t$ asymptotically almost surely.

\begin{theorem}\label{Theorem: Main resul Krull}
Let $I(\mathcal{G})$ be a random monomial ideal with Erdős–Rényi distribution and let $t>1$ be a positive integer. Then $\frac{1}{n^{2/(t-1)}}$ is a threshold function for $q$ for the property $\text{dim}(S/I(\mathcal{G}))\geq t$. In other words
\begin{equation}
\lim_{n\rightarrow \infty}\PP(\text{dim}(S/I(\mathcal{G}))\geq t)= \left\{ \begin{array}{lcc} 0 & \text{if} & q=o(\frac{1}{n^{2/(t-1)}}) \\ \\ 1 & \text{if} & q=\omega(\frac{1}{n^{2/(t-1)}}) \end{array} \right.
\end{equation}
\end{theorem}

\begin{corollary}\label{Corollary: Main theorem Krull}
Let $I(\mathcal{G})$ be a random monomial ideal with Erdős–Rényi distribution and let $t>1$ be a positive integer. If $q=o(\frac{1}{n^{2/t}})$ and $q=\omega(\frac{1}{n^{2/(t-1)}})$ then $\text{dim}(S/I(\mathcal{G}))$ equals $t$ asymptotically almost surely. In other words
\begin{equation}
\lim_{n\rightarrow \infty}\PP(\text{dim}(S/I(\mathcal{G}))= t)=1.
\end{equation}
\end{corollary}

\textit{(C) Regularity.} Let $I\subset S$ be a graded ideal and let ${\mathbf F}$ be the minimal graded free resolution of $S/I$ as an $S$-module:
\[
{\mathbf F}:\ \ \ 0\longrightarrow
\bigoplus_{j}S(-j)^{b_{g,j}}
\stackrel{\varphi_g}{\longrightarrow} \cdots
\longrightarrow\bigoplus_{j}
S(-j)^{b_{1,j}}\stackrel{\varphi_1}{\longrightarrow} S
\longrightarrow S/I \longrightarrow 0.
\]
\quad The ideal $I$ has a $d$-\textit{linear
resolution} if
all maps $\varphi_i$, $i\geq 2$, are defined by matrices whose
entries are linear forms and all entries of $\varphi_1$ are forms of degree
$d$. The \textit{projective dimension} of $S/I$, denoted ${\rm
pd}_S(S/I)$, is equal to $g$. 
The {\it Castelnuovo--Mumford regularity\/} of $S/I$ ({\it
regularity} of $S/I$ for short) is defined as
$${\rm reg}(S/I)=\max\{j-i \mid b_{i,j}\neq 0\}.
$$

By a result of R. Villarreal and Delio \cite{v-number} the regularity is a lower bound for the dimension of $S/I$ provided $I$ is a squarefree monomial ideal.

\begin{proposition}{\rm(\cite[Proposition~3.2]{v-number})} \label{apr15-02}
If $I$ is a squarefree monomial ideal of $S$, then ${\rm reg}(S/I)\leq {\rm dim}(S/I)$.    
\end{proposition}

\quad For our third result we prove for any $t>1$ positive integer the regularity of the ideal is less than or equal to $t-1$ as long as $q=o(\frac{1}{n^{2/(t-1)}})$. The proof of this results follows as a Corollary of Theorem \ref{Theorem: Main resul Krull} combined with Proposition \ref{apr15-02}.

\begin{corollary}\label{Corollary: Main result regularity}
Let $I(\mathcal{G})$ be a random monomial ideal with Erdős–Rényi distribution and let $t>1$ be a positive integer. If $q=o(\frac{1}{n^{2/(t-1)}})$ then $\text{reg}(S/I(\mathcal{G}))\leq t-1$ asymptotically almost surely. In other words
\begin{equation*}
\lim_{n\rightarrow\infty} \PP(\text{reg}(S/I(\mathcal{G}))\leq t-1)=1.
\end{equation*}
\end{corollary}

\textit{(C) $v$-number.} The v-{\em number} of an graded ideal $I$, denoted ${\rm v}(I)$, is the following invariant of
$I$ \cite[Corollary~4.7]{min-dis-generalized}:
$$
{\rm v}(I):=\min\{d\geq 0 \mid\exists\, f 
\in S_d \mbox{ and }\mathfrak{p} \in {\rm Ass}(I) \mbox{ with } (I\colon f)
=\mathfrak{p}\}.
$$
\quad 
One can define the v-number of $I$ locally at each associated 
prime $\mathfrak{p}$ of $I$\/:
$$
{\rm v}_{\mathfrak{p}}(I):=\mbox{min}\{d\geq 0\mid \exists\, f\in S_d
\mbox{ with }(I\colon f)=\mathfrak{p}\}.
$$ 

By a result of R. Villarreal and Delio \cite{v-number} the v-{\em number} of the edge ideal of a clutter $\mathcal{C}$ is a lower bound for the independence number of $\mathcal{C}$.

\begin{corollary}{\rm(\cite[Corollary~3.6]{v-number})}
\label{apr16-02}
If $\mathcal{C}$ is a clutter, then 
\begin{center}
${\rm v}(I(\mathcal{C}))\leq \beta_0(\mathcal{C})$.      
\end{center}
\end{corollary}

For our last result we explore the $v$-number of the ideal $I(\mathcal{G})$. As with the regularity, we prove for any $t>1$ integer the $v$-number of the ideal is less than or equal to $t-1$ as long as $q=o(\frac{1}{n^{2/(t-1)}})$. The proof of this results follows as a Corollary of Theorem \ref{Theorem: Main resul Krull} combined with Corollary \ref{apr16-02}.

\begin{corollary}\label{Corollary: Main result v number}
Let $I(\mathcal{G})$ be a random monomial ideal with Erdős–Rényi distribution and let $t>1$ be a positive integer. If $q=o(\frac{1}{n^{2/(t-1)}})$ then $v(I(\mathcal{G}))\leq t-1$ asymptotically almost surely. In other words
\begin{equation*}
\lim_{n\rightarrow\infty} \PP(v(I(\mathcal{G}))\leq t-1)=1.
\end{equation*}
\end{corollary}

\subsection{proof strategy.} Let us describe our proof techniques. In Theorem \ref{Theorem: Main theorem normality 1} we use the equality of events
$$\{\mathcal{G}\text{ is not normal}\}=\{\mathcal{G}\text{ has a Hochster configuration}\}.$$
Roughly speaking a Hochster configuration consist of two cycles of odd size whose vertices have no edges in between. We then focus our attention into the probability of the second event for which we obtain upper and lower bounds using the following inclusion of events.
\begin{multline*}
\{\mathcal{G}\text{ has }T\text{ as induced subgraph}\}\subset \{\mathcal{G}\text{ has a Hochster configuration}\} \\
\subset \{\mathcal{G}\text{ has a cycle}\}.
\end{multline*}
Here $T$ is the smallest Hochster configuration consisting on two disconnected triangles. To prove our limit theorems we are reduced to prove
$$\PP(\{\mathcal{G}\text{ has }T\text{ as induced subgraph}\})\rightarrow 1,$$
provided $pq^{3/2}=\omega(\frac{1}{n})$ and 
$$\PP(\{\mathcal{G}\text{ has a cycle}\})\rightarrow 0,$$
provided $p=o(\frac{1}{n})$. The second limit is well know however we provide a proof (Proposition \ref{Proposition: having a cycle}) for completeness. The first limit can be computed providing lower bounds for $\PP(\{\mathcal{G}\text{ has }T\text{ as induced subgraph}\})$. The lower bound we provide is of the form
$$\PP(\{\mathcal{G}\text{ has }T\text{ as induced subgraph}\})\geq 1-\frac{Var(Y_{T}(n,p))}{\mathbb{E}(Y_T(n,p))^2},$$
where $Y_T(n,p)$ is the random variable that counts how many times $T$ appears as an induced subgraph of $\mathcal{G}$. We compute upper bounds for the variance of $Y_T(n,p)$ from where we are able to provide an explicit lower bound for $\PP(\{\mathcal{G}\text{ has }T\text{ as induced subgraph}\})$ which we prove converges to $1$ as $n\rightarrow\infty$ provided $pq^{3/2}=\omega(\frac{1}{n})$.

To prove Theorem \ref{Theorem: Main theorem normality 3} we use the inclusion of events

$$\{\mathcal{G}\text{ has no cycles}\}\subset \{\mathcal{G}\text{ is bipartite}\}\subset \{I_c(\mathcal{G})\text{ is normal}\}$$

where we use that the vertex cover ideals of bipartite graphs are normal because they are perfect graphs (see for instance \cite{RAPC}). The limit of the probability of the event on the left is $1$ provided $p=o(\frac{1}{n})$ (Proposition \ref{Proposition: having a cycle}) which forces the probability of the event on the right to converge to $1$.
To prove Theorem \ref{Theorem: Main theorem normality 2} we use Duality criterion \cite[Theorem~5.11]{icdual} which give the inclusion of sets 
\begin{multline*}
\{I_c(\mathcal{G})\text{ is not normal},\beta_0(\mathcal{G})\leq 2\} \\
\subset \{\bar{\mathcal{G}}\text{ has a Hochster configuration}\} \subset
\{\bar{\mathcal{G}}\text{ has a cycle}\}.
\end{multline*}
Here $\bar{\mathcal{G}}$ is the complement of the graph $\mathcal{G}$ defined as the graph such that $\mathcal{G}\cup \bar{\mathcal{G}}$ is the complete graph. Then we are reduced to prove $\PP(\bar{\mathcal{G}}\text{ has a cycle})\rightarrow 0$ which follows as in Theorem \ref{Theorem: Main theorem normality 1}. 

To prove Theorem \ref{Theorem: Main resul Krull} we use the equality of events
$$\{dim(S/I(\mathcal{G}))\geq t\}=\{\mathcal{G}\text{ has }E_t\text{ as induced subgraph}\}.$$
Here $E_t$ is the empty graph (no edges) with $t$ vertices. The rest of the proof consists in proving that $q=\frac{1}{n^{2/(t-1)}}$ is a threshold function for the property $\mathcal{G}$ has $E_t$ as induced subgraph. To prove this it is enough to upper and lower bound the probability of the event
$$\{\mathcal{G}\text{ has }E_t\text{ as an induced subgraph}\}.$$
This event is the same as $\{Y_{E_t}(n,p)>0\}$ whose probability is upper bounded by $\mathbb{E}(Y_{E_t}(n,p))$ and lower bounded by
$$1-\frac{Var_{E_t}(n,p)}{\mathbb{E}(Y_{E_t}(n,p))^2}.$$
Here $Y_{E_t}(n,p)$ is the random variable that counts how many times $E_t$ appears as induced subgraph of $\mathcal{G}$. We find explicit expression for the expectation and bounds for the variance from where we conclude that these bounds converge to $0$ and $1$ respectively if $q=o(\frac{1}{n^{2/(t-1)}})$ and $q=\omega(\frac{1}{n^{2/(t-1)}})$ respectively.

Finally Corollaries \ref{Corollary: Main theorem Krull}, \ref{Corollary: Main result regularity} and \ref{Corollary: Main result v number} follow as a consequence of Theorem \ref{Theorem: Main resul Krull}.

\subsection{Organization of the paper.} Besides Sections \ref{Section: introduction} and \ref{Section: model and results} where we provide an introduction, introduce our model and present our results, the paper is organized as follows. In Section \ref{Section: Monomial ideals} we give a brief introduction to the algebra setting needed for this paper. We introduce the set of monomial ideals generated by graphs and elaborate on some of their properties. In Section \ref{Section: Erdos Renyi graph} we introduce the Erdős–Rényi model. This Section is completely devoted to this model where we explore some threshold functions and find explicit probabilities for some events. Finally, in Section \ref{Section: Proof of main results} we prove our results with the aid of the results previously obtained in Sections \ref{Section: Monomial ideals} and \ref{Section: Erdos Renyi graph}.

\section{Monomial ideals}\label{Section: Monomial ideals}
Let $S=K[x_1,\ldots,x_n]$ be a polynomial ring over
a field $K$ and let $\mathcal{C}$ be a
\textit{clutter} with vertex 
set $V(\mathcal{C})=\{x_1,\ldots,x_n\}$, that is, $\mathcal{C}$ is a family of subsets of
$V(\mathcal{C})$, called \textit{edges}, none of which is included in
another. The set of edges of $\mathcal C$ is denoted by
$E(\mathcal{C})$. The primer example of a clutter is a simple graph
$G$. The monomials of $S$ are denoted 
\begin{center}
$x^a:=x_1^{a_1}\cdots x_n^{a_n}$, {\rm  } $a=(a_1,\dots,a_n)$.
\end{center}

A subset $A$ of $V(\mathcal{C})$ is called 
{\it independent\/} or {\it
stable\/} if $e\not\subset A$ for any  
$e\in E(\mathcal{C})$. The dual concept of a stable vertex set
is a {\it vertex cover\/}, i.e., a subset $C$ of $V(\mathcal{C})$ is a vertex
cover if and only if $V(\mathcal{C})\setminus C$ is a stable vertex set. A 
{\it minimal vertex cover\/} is a vertex cover which
is minimal with respect to inclusion.  
If $A$ is a stable set of vertices of $\mathcal{C}$, the \textit{neighbor set} of $A$, denoted 
$N_\mathcal{C}(A)$, is the set of all vertices $x_i$ 
such that $\{x_i\}\cup A$ contains an edge of $\mathcal{C}$. 
 The number of
vertices in any smallest vertex cover of $\mathcal{C}$, denoted 
$\alpha_0({\mathcal C})$, is called the \textit{vertex covering
number} of $\mathcal{C}$. The \textit{independence
number} of $\mathcal{C}$, denoted by $\beta_0(\mathcal{C})$,
is the number 
of vertices in 
any largest stable set of vertices of $\mathcal{C}$. 
The Krull dimension of $S/I(\mathcal{C})$,
denoted $\dim(S/I(\mathcal{C}))$, is equal to $\beta_0(\mathcal{C})$
and the height of $I(\mathcal{C})$, denoted ${\rm ht}(I(
\mathcal{C}))$, is equal to $\alpha_0(C)$.

\quad Let $C_1,\ldots,C_r$ be the minimal vertex covers
of a graph, $G$, and let $x_{c_1},\ldots,x_{c_r}$ be their corresponding
monomials, that is, $x_{c_j}=\prod_{x_i\in C_j}x_i$ for
$j=1,\ldots,r$.
The following squarefree monomial ideals are associated to $G$.

\begin{definition}\label{Definition: The edge and ideal of covers}
Let $G=(V,E)$ be a graph whose vertices are indexed by the variables $x_1,\dots,x_n$. The \textit{edge ideal} of $G$, denoted $I(G)$, is the
ideal \cite{cmg}:
$$
I(G):=(\{x_ix_j\mid \{x_i,x_j\}\in E(G)\})\subset S.
$$
The \textit{ideal of
covers} of $G$, denoted $I_c(G)$, is the ideal \cite{alexdual}:  
$$
I_c(G):=(x_{c_1},\ldots,x_{c_r})\subset S.
$$
\end{definition}
Any squarefree monomial ideal is
the edge ideal of a clutter \cite[pp.~220--221]{monalg-rev}.

\begin{example}\label{Example: ideal of edges and convers}
Let $G$ be the graph of figure \ref{figure6}. The \textit{edge ideal} of $G$, is the
ideal:
$$
I(G)=(x_1x_2,x_2x_3,x_3x_4,x_4x_5,x_5x_1,x_2x_5,x_3x_5)
$$
The \textit{ideal of covers} of $G$, is the ideal:  
$$
I_c(G)=(x_1x_2x_3x_4,x_1x_3x_5,x_2x_3x_5,x_2x_4x_5).
$$
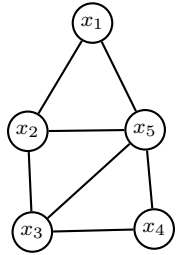
\begin{figure}[ht]
\begin{tikzpicture}[scale=2,thick]
		\tikzstyle{every node}=[minimum width=0pt, inner
		sep=1.8pt, circle]
			\draw (-0.24,1.47) node[draw] (1) { \tiny $x_1$};
			\draw (-0.67,0.75) node[draw] (2) { \tiny $x_2$};
			\draw (-0.64,0.08) node[draw] (3) { \tiny $x_3$};
			\draw (0.17,0.1) node[draw] (4) { \tiny $x_4$};
            \draw (0.11,0.76) node[draw] (5) { \tiny $x_5$};
	
			\draw  (1) edge (2);
			\draw  (3) edge (2);
			\draw  (3) edge (4);
			\draw  (5) edge (4);
			\draw  (5) edge (1);
			\draw  (2) edge (5);
			\draw  (3) edge (5);
		\end{tikzpicture}
\caption{The graph of example \ref{Example: ideal of edges and convers}.}\label{figure6}
\end{figure}
\end{example}


A prime ideal $\mathfrak{p}$ of $S$ is an \textit{associated prime}
of $I$ if
$(I\colon f)=\mathfrak{p}$ for some $f\in S_d$, where $(I\colon f)$ is
the set of all $g\in S$ such that $gf\in I$. The set of associated primes of $I$ 
is denoted by ${\rm Ass}(I)$. 
The following result gives an important relation between the minimal vertex cover of a clutter and its associated primes of $I(\mathcal{C})$.
\begin{lemma}\cite[Lemma~6.3.37]{monalg-rev}\label{jul1-01} Let $C$
be a set of vertices of a clutter
$\mathcal{C}$. Then, $C$ is a minimal vertex cover of $\mathcal{C}$
if and only if the ideal of $S$ generated by $C$ 
is an associated prime of $I(\mathcal{C})$.
\end{lemma}

A \textit{Hochster configuration} of a graph $G$ consists of two odd
cycles $C\sb{1}$, $C\sb{2}$ of $G$ satisfying the following
two conditions:
\begin{enumerate}
\item[(i)] 
$C\sb {1}\cap N_G(C\sb{2})=\emptyset$, where $N_G(C_2)$
is the neighbor set of $C_2$.  
\item[(ii)] No chord of $C\sb {i}$, $i=1,2$, is
an edge of $G$, i.e., $C_i$ is an induced cycle of $G$.
\end{enumerate}

An example of a Hochster configuration can be seen in Figure \ref{figure1}. 

\begin{figure}[ht]
\begin{tikzpicture}[scale=1.5,thick]
		\tikzstyle{every node}=[minimum width=0pt, inner
		sep=1.8pt, circle]
			\draw (-2,1) node[draw] (1) { \tiny $x_1$};
			\draw (-2,-1) node[draw] (2) { \tiny $x_2$};
			\draw (-1,0) node[draw] (3) { \tiny $x_3$};
			\draw (0,0) node[draw] (4) { \tiny $x_4$};
			\draw (1,0) node[draw] (5) { \tiny $x_5$};
			\draw (2,-1) node[draw] (6) { \tiny $x_6$};
			\draw (2,1) node[draw] (7) { \tiny $x_7$};
            
			\draw  (3) edge (1);
			\draw  (1) edge (2);
			\draw  (2) edge (3);
			\draw  (3) edge (4);
			\draw  (4) edge (5);
			\draw  (5) edge (6);
			\draw  (6) edge (7);
			\draw  (7) edge (5);
		\end{tikzpicture}
\caption{A Hochster configuration.}\label{figure1}
\end{figure}
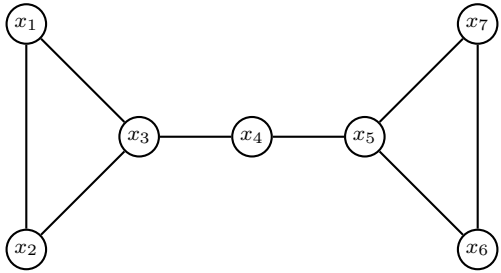

The following result gives a combinatorial description of the normality of the edge ideal.
\begin{theorem}{\rm(\cite[Corollary~5.8.10]{graphs})}
\label{apr16-03} 
The edge ideal $I(G)$ of a 
graph $G$ is
normal if and only if $G$ admits no Hochster configurations.
\end{theorem}

\begin{remark}
With the help of Theorem \ref{apr16-03} we will be able to find the probability of the event
$$\{I(\mathcal{G})\text{ is not normal}\}$$
by looking at the probability of the event
$$\{\mathcal{G}\text{ has a Hochster configuration}\}.$$
We will be able to lower and upper bound this probability with the probability of the events
$$\{\mathcal{G}\text{ has }T\text{ as induced subgraph}\},$$
and 
$$\{\mathcal{G}\text{ has a cycle}\},$$
respectively. Here $T$ is a graph that consist of two disconnected triangles. Then we will prove the probability of these events converge to $0$ and $1$ respectively under suitable conditions for $p$ and $q$.
\end{remark}

The complement of a graph $G$, denoted by $\overline{G}$, is the graph with the same set of vertices and such that two distinct vertices of $\overline{G}$ are adjacent if and only if they are not adjacent in $G$. An example of a graph and its complement can be seen in Figure \ref{figure3}.

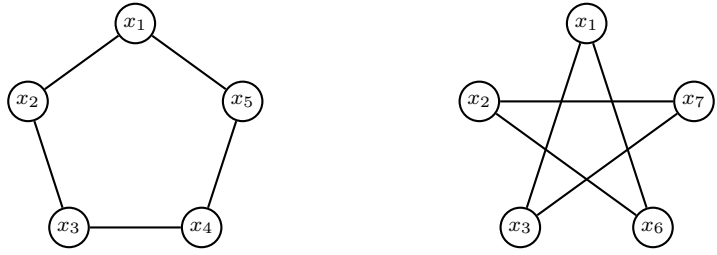
\begin{figure}[ht]
\begin{tikzpicture}[scale=1.5,thick]
		\tikzstyle{every node}=[minimum width=0pt, inner
		sep=1.8pt, circle]
			\draw (-2,1) node[draw] (1) { \tiny $x_1$};
			\draw (-2.9511,0.3090) node[draw] (2) { \tiny $x_2$};
			\draw (-2.5878,-0.8090) node[draw] (3) { \tiny $x_3$};
            \draw (-1.4122,-0.8090) node[draw] (4) { \tiny $x_4$};
			\draw (-1.0489,0.3090) node[draw] (5) { \tiny $x_5$};
            
			\draw  (1) edge (2);
			\draw  (2) edge (3);
			\draw  (3) edge (4);
			\draw  (4) edge (5);
			\draw  (5) edge (1);

            \draw (2,1) node[draw] (11) { \tiny $x_1$};
			\draw (1.0489,0.3090) node[draw] (12) { \tiny $x_2$};
			\draw (1.4122,-0.8090) node[draw] (13) { \tiny $x_3$};
            \draw (2.5878,-0.8090) node[draw] (14) { \tiny $x_6$};
			\draw (2.9511,0.3090) node[draw] (15) { \tiny $x_7$};
            
			\draw  (11) edge (13);
			\draw  (11) edge (14);
			\draw  (15) edge (12);
			\draw  (15) edge (13);
			\draw  (14) edge (12);

		\end{tikzpicture}
\caption{The graph $C_5$ (left) and its complement $\overline{G}$ (right).}\label{figure3}
\end{figure} 

The following result \cite{icdual} gives a combinatorial description of 
the normality of the ideal of covers of graphs 
with independence number at most two.

\begin{theorem}{\rm(\cite[Theorem~5.11]{icdual})}(Duality
criterion) \label{Duality criterion}
Let $G$ be a graph with $\beta_0(G)\leq 2$. The
following hold.
\begin{enumerate}
\item[(a)] $I_c(G)$ is normal if and only if $I(\overline{G})$ is normal. 
\item[(b)] $I_c(G)$ is normal if and only if $\overline{G}$ has no
Hochster configurations.
\end{enumerate}    
\end{theorem}

\section{Erdős–Rényi graphs and threshold functions}\label{Section: Erdos Renyi graph}

In this section we introduce the Erdős–Rényi random graph model. Further we investigate threshold functions for some properties of the Erdős–Rényi graph. 

We begin with a brief introduction to the definitions and notations used in this
paper regarding graphs. A \textit{graph} is a pair $G=(V,E)$ where $V$ is the set of vertices and $E\subset V\times V$ are the edges of the graph. Given a graph $G=(V,E)$ and a subset $V^\prime \subset V$ of $V$ we let $G_{V^\prime}$ to be the restriction of $G$ to $V^\prime$ defined as the graph with vertex set $V^\prime$ and edges $\{u,v\}$ for all $\{u,v\}\in E$ with $u,v\in V^\prime$. We say that $G^\prime=(V^\prime,E^\prime)$ is an \textit{induced subgraph} of $G$ if $V^\prime\subset V$ and $\{u,v\}\in E^\prime$ if and only if $\{u,v\}\in E$ for all $u,v\in V^\prime$. An example of a graph and and induced subgraph can be seen in Figure \ref{figure2}.

Two graphs $G_1=(V_1,E_1)$ and $G_2=(V_2,E_2)$ are considered the same if they are isomorphic, that is, there exist a bijection $f:V_1\rightarrow V_2$ and $(u,v)\in E_1$ if and only if $(f(u),f(v))\in E_2$ for all $u,v\in V_1$.

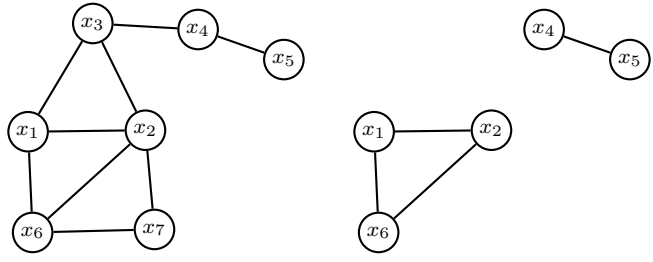
\begin{figure}[ht]
\begin{tikzpicture}[scale=2,thick]
		\tikzstyle{every node}=[minimum width=0pt, inner
		sep=1.8pt, circle]
			\draw (-2.97,0.75) node[draw] (1) { \tiny $x_1$};
			\draw (-2.19,0.76) node[draw] (2) { \tiny $x_2$};
			\draw (-2.54,1.47) node[draw] (3) { \tiny $x_3$};
            \draw (-1.84,1.43) node[draw] (4) { \tiny $x_4$};
			\draw (-1.27,1.23) node[draw] (5) { \tiny $x_5$};
            \draw (-2.94,0.08) node[draw] (6) { \tiny $x_6$};
			\draw (-2.13,0.1) node[draw] (7) { \tiny $x_7$};
            
			\draw  (5) edge (4);
			\draw  (3) edge (4);
			\draw  (3) edge (1);
			\draw  (3) edge (2);
			\draw  (1) edge (2);
			\draw  (1) edge (6);
			\draw  (2) edge (7);
			\draw  (6) edge (7);
            \draw  (6) edge (2);

            \draw (-0.67,0.75) node[draw] (11) { \tiny $x_1$};
			\draw (0.11,0.76) node[draw] (12) { \tiny $x_2$};
            \draw (0.46,1.43) node[draw] (14) { \tiny $x_4$};
			\draw (1.03,1.23) node[draw] (15) { \tiny $x_5$};
            \draw (-0.64,0.08) node[draw] (16) { \tiny $x_6$};
            
			\draw  (15) edge (14);
			\draw  (11) edge (12);
			\draw  (11) edge (16);
            \draw  (16) edge (12);

		\end{tikzpicture}
\caption{A graph $G$ (left) and an induced subgraph of $G$ (right). Here $V^\prime=\{x_1,x_2,x_4,x_5,x_6\}$. Oberve that the graph on the right is also the restriction of $G$ to $V^\prime$, $G_{V^\prime}$.}\label{figure2}
\end{figure}

\begin{notation}\label{Notation: The graphs T and Ct}
Let $t>1$ be an integer. We define the graphs $T$ and $E_t$ as follows.
\begin{enumerate}
    \item $T$ is the graph with vertex set $V=\{1,\dots,6\}$ and edges $E=\{(1,2),(2,3),(3,1),(4,5),(5,6),(6,3)\}$.
    \item $E_t$ is the graph with vertex set $V=\{1,\dots,t\}$ and edges $E=\emptyset$.
\end{enumerate}
Examples of $T$ and $E_t$ can be seen in Figure \ref{figure4}.
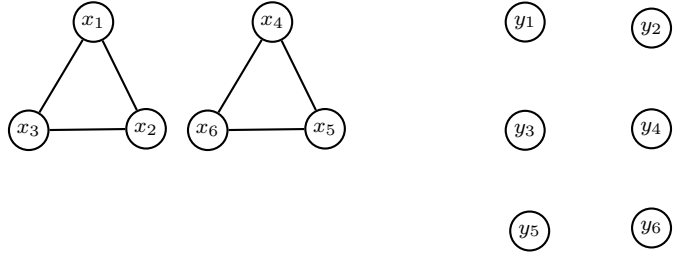
\begin{figure}[ht]
\begin{tikzpicture}[scale=2,thick]
		\tikzstyle{every node}=[minimum width=0pt, inner
		sep=1.8pt, circle]
            \draw (-3.54,1.47) node[draw] (1) { \tiny $x_1$};
			\draw (-3.19,0.76) node[draw] (2) { \tiny $x_2$};
			\draw (-3.97,0.75) node[draw] (3) { \tiny $x_3$};
			\draw (-2.35,1.47) node[draw] (4) { \tiny $x_4$};
			\draw (-2,0.76) node[draw] (5) { \tiny $x_5$};
			\draw (-2.78,0.75) node[draw] (6) { \tiny $x_6$};
            
			\draw  (1) edge (2);
			\draw  (2) edge (3);
			\draw  (3) edge (1);
			\draw  (4) edge (5);
			\draw  (5) edge (6);
			\draw  (6) edge (4);

			\draw (-0.67,1.47) node[draw] (11) { \tiny $y_1$};
            \draw (0.17,1.43) node[draw] (12) { \tiny $y_2$};
            \draw (-0.67,0.75) node[draw] (13) { \tiny $y_3$};
			\draw (0.17,0.76) node[draw] (14) { \tiny $y_4$};
            \draw (-0.64,0.08) node[draw] (15) { \tiny $y_5$};
			\draw (0.17,0.1) node[draw] (16) { \tiny $y_6$};
            
		\end{tikzpicture}
\caption{The graph $T$ (left) and the graph $E_t$ (right) with $t=6$.}\label{figure4}
\end{figure}    
\end{notation}

We are now in place to introduce the Erdős–Rényi model. We say that a random graph $\mathcal{G}$ has Erdős–Rényi distribution with parameters $n\in \mathbb{N}$ and $0\leq p\leq 1$, denoted $\mathcal{G}\sim ER(n,p)$, if $V=[n]:=\{1,\dots,n\}$ and each edge $\{u,v\}$ is chosen with probability $p$ for all $1\leq u,v\leq n$ such that $u\neq v$.

\begin{definition}
Let $n\in\mathbb{N}$, $0\leq p\leq 1$ and $\mathcal{G}\sim ER(n,p)$. Let $G$ be a graph, we define the random variable $Y_G(n,p)$ by
$$Y_G(n,p)=\sum_{V\subset [n]} \mathbbm{1}_{\{\mathcal{G}_V=G\}}.$$
$Y_G(n,p)$ counts the number of times that $G$ appears as an induced subgraph of $\mathcal{G}$.
\end{definition}

\begin{proposition}\label{Proposition: Moments of T}
Let $n\in\mathbb{N}$ with $n\geq 6$, $0\leq p\leq 1$ and $\mathcal{G}\sim ER(n,p)$. Let $T$ be the graph defined in notation \ref{Notation: The graphs T and Ct}. Then
\begin{equation}\label{Equation: Expectation of T}
\mathbb{E}(Y_T(n,p))=\binom{n}{6}\binom{6}{3}p^6(1-p)^9
\end{equation}
and,
\begin{multline}\label{Equation: Variance of T}
\V(Y_T(n,p))\leq n^{10}p^{11}q^{18}+n^{10}p^{12}q^{17}+n^{9}p^{11}q^{16}+n^9p^9q^{18} \\
+n^8p^{10}q^{14}+n^8p^9q^{15}+n^7p^8q^{12}+n^6p^6q^9
\end{multline}
\end{proposition}
\begin{proof}
We have
$$\mathbb{E}(Y_T(n,p))=\sum_{V\subset [n]}\mathbb{E}(\mathbbm{1}_{\{\mathcal{G}_V=T\}})=\sum_{V\subset [n]}\PP(\mathcal{G}_V=T).$$
This probability is non zero only when $|V|=6$, further in that case the probability is given by $\binom{6}{3}p^6q^9$ as we choose first three vertices that will conform one of the triangles of $T$ and then we multiply by the probability of having $6$ edges (corresponding to the two triangles of $T$) and the probability of all the other possible edges being empty. Finally there are $\binom{n}{6}$ ways of choosing $V$, this proves (\ref{Equation: Expectation of T}). 
Now we prove (\ref{Equation: Variance of T}). Observe that
\begin{eqnarray*}
&&\V(Y_T(n,p))=\mathbb{E}(Y_T(n,p)^2)-\mathbb{E}(Y_T(n,p))^2 \\
&=& \mathbb{E}\left[\sum_{V\subset [n]}\sum_{W\subset [n]}\mathbbm{1}_{\{\mathcal{G}_V=T\}}\mathbbm{1}_{\{\mathcal{G}_W=T\}}\right]-\sum_{V\subset [n]}\sum_{W\subset [n]}\mathbb{E}(\mathbbm{1}_{\{\mathcal{G}_V=T\}})\mathbb{E}(\mathbbm{1}_{\{\mathcal{G}_W=T\}}) \\
&=& \sum_{V\subset [n]}\sum_{W\subset [n]}\mathbb{E}(\mathbbm{1}_{\{\mathcal{G}_V=T\}}\mathbbm{1}_{\{\mathcal{G}_W=T\}})-\mathbb{E}(\mathbbm{1}_{\{\mathcal{G}_V=T\}})\mathbb{E}(\mathbbm{1}_{\{\mathcal{G}_W=T\}})
\end{eqnarray*}
Observe that when $V$ and $W$ have at most one vertex in common then $\{\mathcal{G}_V=T\}$ and $\{\mathcal{G}_W=T\}$ are independent events and therefore 
$$\mathbb{E}(\mathbbm{1}_{\{\mathcal{G}_V=T\}}\mathbbm{1}_{\{\mathcal{G}_W=T\}})-\mathbb{E}(\mathbbm{1}_{\{\mathcal{G}_V=T\}})\mathbb{E}(\mathbbm{1}_{\{\mathcal{G}_W=T\}})=0.$$
Thus
\begin{eqnarray}\label{Equation: Aux1}
\V(Y_T(n,p))&=&\sum_{\substack{V,W\subset [n] \\ |V\cap W|\geq 2}}\mathbb{E}(\mathbbm{1}_{\{\mathcal{G}_V=T\}}\mathbbm{1}_{\{\mathcal{G}_W=T\}})-\mathbb{E}(\mathbbm{1}_{\{\mathcal{G}_V=T\}})\mathbb{E}(\mathbbm{1}_{\{\mathcal{G}_W=T\}}) \nonumber \\
&\leq & \sum_{\substack{V,W\subset [n] \\ |V\cap W|\geq 2}}\mathbb{E}(\mathbbm{1}_{\{\mathcal{G}_V=T\}}\mathbbm{1}_{\{\mathcal{G}_W=T\}}) \nonumber\\
&=& \sum_{\substack{V,W\subset [n] \\ |V\cap W|\geq 2}}\PP(\mathcal{G}_V=T ,\mathcal{G}_W=T)
\end{eqnarray}
We now compute the probabilities $\PP(\mathcal{G}_V=T ,\mathcal{G}_W=T)$ for $|V\cap W|=1,\dots,6$. If $|V\cap W|=2$ there are two possible scenarios, either the shared vertices have an edge or they are disconnected. In the first scenario we claim the probability is  $16p^{11}q^{18}$. Indeed, we first choose one vertex from $V$ that will conform a triangle with the two shared edges (see Figure \ref{figure5}). This can be done in $4$ ways, the same for choosing the ones in $W$. From here we get $16$. Secondly, there are $3$ possible types of pair of vertices; the ones that must be connected, the ones that must be disconnected and the ones that can be either connected or disconnected. There are $11$ edges in $V\cup W$ which must be in the graph from where we get $p^{11}$, this correspond to the pair of vertices that must be connected (see black edges in Figure \ref{figure5}). We require $\mathcal{G}_V=T$ therefore $9$ edges must be empty, these correspond to the pair of vertices that must be disconnected (see the green edge in Figure \ref{figure5} for an example). In the same way in $\mathcal{G}_W$ there must be $9$ empty edges, from where we get $q^{18}$. Finally there are pair of vertices that might be connected or disconnected and in both cases the event $\{\mathcal{G}_V=T,\mathcal{G}_W=T\}$ holds (these correspond to edges as the one in blue in Figure \ref{figure5}). So we multiply by $1$ corresponding to $p+q$. This gives a total probability of $16p^{11}q^{18}$.

\begin{figure}[ht]
\begin{tikzpicture}[scale=1.5,thick]
		\tikzstyle{every node}=[minimum width=0pt, inner
		sep=1.8pt, circle]
			\draw (-2,1) node[draw, circle, fill=gray] (1) {};
			\draw (-2,0) node[draw, circle, fill=gray] (2) {};
			\draw (-1,0) node[draw, circle, fill=gray] (3) {};
            \draw (1,0) node[draw, circle, fill=blue] (4) {};
			\draw (2,0) node[draw, circle, fill=blue] (5) {};
            \draw (2,1) node[draw, circle, fill=blue] (6) {};
			\draw (1,2) node[draw, circle, fill=blue] (7) {};
            \draw (0,2) node[draw, circle, fill=red] (8) {};
            \draw (-1,2) node[draw, circle, fill=gray] (9) {};
            \draw (0,1) node[draw, circle, fill=red] (10) {};
            
			\draw  (1) edge (2);
			\draw  (3) edge (2);
			\draw  (3) edge (1);
			\draw  (4) edge (5);
			\draw  (6) edge (5);
			\draw  (4) edge (6);
			\draw  (10) edge (7);
			\draw  (8) edge (7);
            \draw  (8) edge (9);
            \draw  (10) edge (9);
			\draw  (10) edge (8);
            \draw[color=green]  (1) edge (9);
            \draw[color=blue]  (3) edge (4);

		\end{tikzpicture}
\caption{The sets $V$ (grey and red vertices) and $W$ (blue and red vertices). In red the shared vertices of $V$ and $W$.}\label{figure5}
\end{figure}
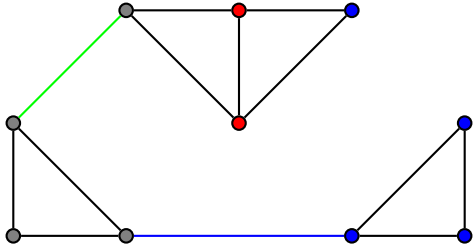

\quad In the second scenario we proceed similarly, now the probability is given by $36p^{12}q^{17}$. Observe that in this case we obtain $q^{17}$ instead of $q^{18}$ because the two shared vertices must be disconnected in both graphs ($\mathcal{G}_V$ and $\mathcal{G}_W$), therefore the total number of edges that must be empty in the graph is $9+9-1$. Finally observe that the number of ways of choosing $V$ and $W$ is given by
$$\binom{n}{6}\binom{6}{2}\binom{n-6}{4}.$$
We conclude that
\begin{eqnarray*}
\sum_{\substack{V,W\subset [n] \\ |V\cap W|= 2}}\PP(\mathcal{G}_V=T ,\mathcal{G}_W=T) &=& \binom{n}{6}\binom{6}{2}\binom{n-6}{4}(16p^{11}q^{18}+36p^{12}q^{17}) \\
&\leq & n^{10}p^{11}q^{18}+n^{10}p^{12}q^{17}.
\end{eqnarray*}
We proceed similarly for the cases $|V\cap W|=2,\dots,6$ for which we get
\begin{eqnarray*}
\sum_{\substack{V,W\subset [n] \\ |V\cap W|= 3}}\PP(\mathcal{G}_V=T ,\mathcal{G}_W=T) &=& \binom{n}{6}\binom{6}{3}\binom{n-6}{3}(27p^{11}q^{18}+p^{9}q^{18}) \\
&\leq & n^{9}p^{11}q^{16}+n^{9}p^{9}q^{18},
\end{eqnarray*}
\begin{eqnarray*}
\sum_{\substack{V,W\subset [n] \\ |V\cap W|= 4}}\PP(\mathcal{G}_V=T ,\mathcal{G}_W=T) &=& \binom{n}{6}\binom{6}{4}\binom{n-6}{2}(24p^{10}q^{14}+4p^{9}q^{15}) \\
&\leq & n^{8}p^{10}q^{14}+n^{8}p^{9}q^{15},
\end{eqnarray*}
\begin{eqnarray*}
\sum_{\substack{V,W\subset [n] \\ |V\cap W|= 5}}\PP(\mathcal{G}_V=T ,\mathcal{G}_W=T) &=& \binom{n}{6}\binom{6}{5}\binom{n-6}{1}(10p^{8}q^{12}) \\
&\leq & n^{7}p^{8}q^{12},
\end{eqnarray*}
and 
\begin{eqnarray*}
\sum_{\substack{V,W\subset [n] \\ |V\cap W|= 6}}\PP(\mathcal{G}_V=T ,\mathcal{G}_W=T) &=& \binom{n}{6}(20p^{6}q^{9}) \\
&\leq & n^{6}p^{6}q^{9}.
\end{eqnarray*}
Combining these inequalities with Inequality (\ref{Equation: Aux1}) concludes the proof.
\end{proof}

\begin{corollary}\label{Corollary: Lower bound of T}
Let $n\in\mathbb{N}$ with $n\geq 6$, $0\leq p\leq 1$ and $\mathcal{G}\sim ER(n,p)$. Let $T$ be the graph defined in notation \ref{Notation: The graphs T and Ct}. Then
\begin{multline}
\PP(Y_T(n,p)>0) \geq 1-\frac{1}{400\binom{n}{6}^2p^{12}q^{18}}(n^{10}p^{11}q^{18}+n^{10}p^{12}q^{17}+n^{9}p^{11}q^{16} \\ 
+n^9p^9q^{18}+n^8p^{10}q^{14}+n^8p^9q^{15}+n^7p^8q^{12}+n^6p^6q^9)
\end{multline}
\end{corollary}
\begin{proof}
Note that if $\mathbb{E}(Y_T(n,p))>|Y_T(n,p)-\mathbb{E}(Y_T(n,p))|$ then $Y_T(n,p)>0$, thus
\begin{eqnarray*}
\PP(Y_T(n,p)>0) &\geq & \PP(\mathbb{E}(Y_T(n,p))>|Y_T(n,p)-\mathbb{E}(Y_T(n,p))|) \\
&=& 1-\PP(|Y_T(n,p)-\mathbb{E}(Y_T(n,p))|\geq \mathbb{E}(Y_T(n,p))) \\
&\geq & 1-\frac{Var(Y_T(n,p))}{\mathbb{E}(Y_T(n,p))^2}.
\end{eqnarray*}
Where the last inequality follows from Chebyshev inequality. We conclude using Equation (\ref{Equation: Expectation of T}) and Inequality (\ref{Equation: Variance of T}).
\end{proof}

With the help of Propositions \ref{Proposition: Moments of T} and Corollary \ref{Corollary: Lower bound of T} we get the following threshold function.

\begin{corollary}
Let $T$ be the graph defined in notation \ref{Notation: The graphs T and Ct}. Let $n\in\mathbb{N}$ with $n\geq 6$, $0\leq p\leq 1$ and $\mathcal{G}\sim ER(n,p)$. Then $\frac{1}{n}$ is a threshold function for $pq^{3/2}$ for the property $\mathcal{G}$ has $T$ as an induced subgraph. In other words
\begin{equation*}
\lim_{n\rightarrow \infty}\PP(\mathcal{G}\text{ has }T\text{ as an induced subgraph})= \left\{ \begin{array}{lcc} 0 & \text{if} & pq^{3/2}=o(\frac{1}{n}) \\ \\ 1 & \text{if} & pq^{3/2}=\omega(\frac{1}{n}) \end{array} \right.
\end{equation*}
\end{corollary}
\begin{proof}
Note that $\{\mathcal{G}\text{ has }T\text{ as an induced subgraph}\}=\{Y_T(n,p)>0\}$ hence,
$$\PP(\{\mathcal{G}\text{ has }T\text{ as an induced subgraph}\})=\PP(\{Y_T(n,p)>0\})\leq \mathbb{E}(Y_T(n,p)).$$
From Proposition \ref{Proposition: Moments of T} the limit of the right hand side converges to $0$ given $pq^{3/2}=o(\frac{1}{n})$ which proves the first part. We omit the proof of the second part as it will be exactly the same as that of our Theorem \ref{Theorem: Main theorem normality 1}.
\end{proof}

\begin{proposition}\label{Proposition: Moments of Ct}
Let $t>1$ be an integer and $E_t$ the graph defined in notation \ref{Notation: The graphs T and Ct}. Let $n\in\mathbb{N}$ with $n\geq t$, $0\leq p\leq 1$ and $\mathcal{G}\sim ER(n,p)$. Then
\begin{equation}\label{Equation: Expectation of Ct}
\mathbb{E}(Y_{E_t}(n,p))=\binom{n}{t}(1-p)^{\binom{t}{2}}
\end{equation}
and,
\begin{equation}\label{Equation: Variance of Ct}
V(Y_{E_t}(n,p))\leq n^{2t}(1-p)^{2\binom{t}{2}}\sum_{j=2}^t \frac{1}{n^j(1-p)^{\binom{j}{2}}}.
\end{equation}
\end{proposition}
\begin{proof}
We proceed similarly as in Proposition \ref{Proposition: Moments of T}. We have
$$\mathbb{E}(Y_{E_t}(n,p))=\sum_{V\subset [n]}\mathbb{E}(\mathbbm{1}_{\{\mathcal{G}_V=C_t\}})=\sum_{V\subset [n]}\PP(\mathcal{G}_V=C_t).$$
This probability is non zero only when $|V|=t$, further in that case the probability is given by $q^{\binom{t}{2}}$ as we require the graph to be empty. Finally there are $\binom{n}{t}$ ways of choosing $V$, this proves (\ref{Equation: Expectation of Ct}). To prove (\ref{Equation: Variance of Ct}) we use Equation (\ref{Equation: Aux1}): 
\begin{equation}\label{Equation: Aux2}
Var(Y_{E_t}(n,p)) \leq \sum_{\substack{V,W\subset [n] \\ |V\cap W|\geq 2}}\PP(\mathcal{G}_V=C_t,\mathcal{G}_W=C_t)
\end{equation}
In this case we need to consider $|V\cap W|=2,\dots, t$. for all $V$ and $W$ such that $|V|=|W|=t$. Let $2\leq j\leq t$ and $|V\cap W|=j$, then
$$\PP(\mathcal{G}_V=C_T,\mathcal{G}_W=C_t)=q^{2\binom{t}{2}-\binom{j}{2}},$$
because we require $\mathcal{G}_V$ to be the empty graph which occurs with probability $q^{\binom{t}{2}}$. The same reasoning applies to $\mathcal{G}_W$, however as $|V\cap W|=j$ then we are counting twice the edges that must be empty and whose vertices are in $V\cap W$, there are $\binom{j}{2}$ of these edges. Therefore we multiply by the probability $q^{-\binom{j}{2}}$. Finally there are
$$\binom{n}{t}\binom{t}{j}\binom{n-t}{t-j}$$
ways of choosing $V$ and $W$. Thus from inequality (\ref{Equation: Aux2})
\begin{eqnarray*}
Var(T_{E_t}(n,p))&\leq & \sum_{j=2}^t \binom{n}{t}\binom{t}{j}\binom{n-t}{t-j}q^{2\binom{t}{2}-\binom{j}{2}} \\
&\leq & \sum_{j=2}^t n^{2t-j}q^{2\binom{t}{2}-\binom{j}{2}}
\end{eqnarray*}
which proves (\ref{Equation: Variance of Ct}).
\end{proof}

\begin{corollary}\label{Corollary: Lower bound of Ct}
Let $t>1$ be an integer and $E_t$ the graph defined in notation \ref{Notation: The graphs T and Ct}. Let $n\in\mathbb{N}$ with $n\geq t$, $0\leq p\leq 1$ and $\mathcal{G}\sim ER(n,p)$. Then
\begin{equation}
\PP(Y_{E_t}(n,p)>0)\geq 1-\frac{n^{2t}}{\binom{n}{t}^2}\sum_{j=2}^t \frac{1}{n^j(1-p)^{\binom{j}{2}}}.
\end{equation}
\end{corollary}
\begin{proof}
Proceeding as in Corollary \ref{Corollary: Lower bound of T}
\begin{eqnarray*}
\PP(Y_{E_t}(n,p)>0)&\geq & 1-\frac{Var(T_{E_t}(n,p))}{\mathbb{E}(Y_{E_t}(n,p))^2} \\
&\geq & 1- \frac{n^{2t}(1-p)^{2\binom{t}{2}}\sum_{j=2}^t \frac{1}{n^j (1-p)^{\binom{j}{2}}}}{\binom{n}{t}^2 (1-p)^{2\binom{t}{2}}} \\
&=& 1-\frac{n^{2t}}{\binom{n}{t}^2}\sum_{j=1}^t \frac{1}{n^j (1-p)^{\binom{j}{2}}},
\end{eqnarray*}
where the second inequality follows from Proposition \ref{Proposition: Moments of Ct}. 
\end{proof}

As a Corollary we obtain the following threshold function.

\begin{corollary}
Let $t>1$ be an integer and $E_t$ the graph defined in notation \ref{Notation: The graphs T and Ct}. Let $n\in\mathbb{N}$ with $n\geq t$, $0\leq p\leq 1$ and $\mathcal{G}\sim ER(n,p)$. Then $\frac{1}{n^{2/(t-1)}}$ is a threshold function for $q$ for the property $\mathcal{G}$ has $E_t$ as an induced subgraph. In other words
\begin{equation*}
\lim_{n\rightarrow \infty}\PP(\mathcal{G}\text{ has }E_t\text{ as an induced subgraph})= \left\{ \begin{array}{lcc} 0 & \text{if} & q=o(\frac{1}{n^{2/(t-1)}}) \\ \\ 1 & \text{if} & q=\omega(\frac{1}{n^{2/(t-1)}}) \end{array} \right.
\end{equation*}
\end{corollary}

We omit the proof as it is exactly the same argument that we use to prove our main Theorem \ref{Theorem: Main resul Krull}. 

To conclude this section we prove that if $p=o(\frac{1}{n})$ then $\mathcal{G}\sim ER(n,p)$ has no cycles asymptotically almost surely.

\begin{proposition}\label{Proposition: having a cycle}
Let $n\in\mathbb{N}$, $0\leq p\leq 1$ and $\mathcal{G}\sim ER(n,p)$. If $p=o(\frac{1}{n})$ then
$$\lim_{n\rightarrow\infty}\PP(\{\mathcal{G}\text{ has a cycle}\})=0.$$
\end{proposition}
\begin{proof}
Let $X_n$ the random variable that counts the number of cycles of $\mathcal{G}$. Hence
$$\PP(\{\mathcal{G}\text{ has a cycle}\})=\PP(X_n>0)\leq \mathbb{E}(X_n),$$
where last inequality follows from Markov's inequality. So it is enough to prove $\mathbb{E}(X_n)\rightarrow 0$. Let $S_k\subset [n]$ be the subsets of $[n]$ of size $k$, then
$$X_n=\sum_{k=3}^n \sum_{S\subset S_k}\mathbbm{1}_{\{\mathcal{G}_{S_k}\text{ is a cycle of length } k\}}.$$
Therefore
$$\mathbb{E}(X_n)=\sum_{k=3}^n \sum_{S\subset S_k}\PP({\{\mathcal{G}_{S_k}\text{ is a cycle of length } k\}}).$$
For a fixed $S\subset S_k$ observe that there are at most $\binom{n}{k}k!$ cycles as this counts the number of ordered set of vertices and each ordered set of vertices induces a cycle by joining adjacent vertices. It might be possible we are over-counting cycles, in either case the number of cycles is upper bounded by $\binom{n}{k}k!$. Given an ordered set of vertices the probability of having a cycle is $p^k$, hence
$$\mathbb{E}(X_n)\leq \sum_{k= 3}^n \binom{n}{k}k!p^k \leq \sum_{k\geq 3} n^kp^k = (np)^3\sum_{k\geq 0}(np)^k = \frac{(np)^3}{1-np}.$$
Since the right hand side converges to $0$ the proof is completed.
\end{proof}

\section{Proof of main results}\label{Section: Proof of main results}

This Section is devoted to provided the proof of all of our results. Let us remind that our idea is finding upper and lower bounds for the desired probabilities with the help of Section \ref{Section: Erdos Renyi graph}.

\begin{proof}[Proof of Theorem \ref{Theorem: Main theorem normality 1}]
From \cite[Theorem~5.8]{v-number} we have the following equality of events
$$\{I(\mathcal{G})\text{ is not normal}\}=\{\mathcal{G}\text{ has a Hochster configuration}\}.$$
Firstly, we prove that if $p=o(\frac{1}{n})$ then 
$$\PP(\{\mathcal{G}\text{ has a Hochster configuration}\})\rightarrow 0.$$
This follows easily from the fact that
$$0\leq \PP(\{\mathcal{G}\text{ has a Hochster configuration}\}) \leq \PP(\{\mathcal{G}\text{ has a cycle}\})$$
where the right hand side converges to $0$ because of Proposition \ref{Proposition: having a cycle}.
Secondly, we prove that if $pq^{3/2}=\omega(\frac{1}{n})$ then 
$$\PP(\{\mathcal{G}\text{ has a Hochster configuration}\})\rightarrow 1.$$
Indeed
\begin{equation}\label{Equation: Aux5}
1\geq \PP(\{\mathcal{G}\text{ has a Hochster configuration}\})\geq \PP(Y_{T}(n,p)>0),
\end{equation}
where the last inequality follows because $T$ is a Hochster configuration of the graph. From Corollary \ref{Corollary: Lower bound of T} we know
\begin{multline*}
\PP(Y_T(n,p)>0) \geq 1-\frac{1}{400\binom{n}{6}^2p^{12}q^{18}}(n^{10}p^{11}q^{18}+n^{10}p^{12}q^{17}+n^{9}p^{11}q^{16} \\ 
+n^9p^9q^{18}+n^8p^{10}q^{14}+n^8p^9q^{15}+n^7p^8q^{12}+n^6p^6q^9)
\end{multline*}
Using that $\binom{n}{6} \leq \frac{n^6}{6!}$ and previous inequality yields
\begin{multline}\label{Equation: Aux6}
\lim_{n\rightarrow\infty}\PP(Y_T(n,p)>0) \geq \lim_{n\rightarrow\infty}1-\frac{6!^2}{400n^{12}p^{12}q^{18}}(n^{10}p^{11}q^{18}+n^{10}p^{12}q^{17} \\
+n^{9}p^{11}q^{16}+n^9p^9q^{18}+n^8p^{10}q^{14}+n^8p^9q^{15}+n^7p^8q^{12}+n^6p^6q^9) \\
\geq \lim_{n\rightarrow\infty}1-\frac{1296}{n^2p+n^2q+n^3pq^2+n^3p^3+n^4p^2q^4+n^4p^3q^3+n^5p^4q^6+n^6p^6q^9} \\
= 1
\end{multline}
where the limit equals $1$ because
\begin{eqnarray*}
&&n^2p\geq n^2p^2q^3 = (npq^{3/2})^2 \rightarrow \infty \\
&&n^2q \geq n^2p^2q^3 = (npq^{3/2})^2 \rightarrow \infty \\
&&n^3pq^2 \geq n^3p^3q^{9/2} = (npq^{3/2})^3 \rightarrow \infty \\
&&n^3p^3 \geq n^3p^3q^{9/2} = (npq^{3/2})^3 \rightarrow \infty \\
&&n^4p^2q^4 \geq n^4p^4q^6 = (npq^{3/2})^4 \rightarrow \infty \\
&&n^4p^3q^3 \geq n^3p^3q^{9/2} = (npq^{3/2})^3 \rightarrow \infty \\
&&n^5p^4q^6 \geq n^4p^4q^6 = (npq^{3/2})^4 \rightarrow \infty \\
&&n^6p^6q^9 = (npq^{3/2})^6 \rightarrow \infty.
\end{eqnarray*}
Combining Inequalities (\ref{Equation: Aux5}) and (\ref{Equation: Aux6}) yields the desired result.
\end{proof}

\begin{proof}[Proof of Theorem \ref{Theorem: Main theorem normality 3}]
In \cite{RAPC} it was proved that the cover ideal of a perfect graph is normal. Since bipartite graphs are perfect we have the inclusion of events
$$\{\mathcal{G}\text{ has no cycles}\}\subset \{\mathcal{G}\text{ is bipartite}\}\subset \{I_c(\mathcal{G})\text{ is normal}\}.$$
Therefore,
$$\lim_{n\rightarrow\infty}\PP(\{\mathcal{G}\text{ has no cycles}\})\leq \lim_{n\rightarrow\infty}\PP(I_c(\mathcal{G})\text{ is normal})\leq 1.$$
From Proposition \ref{Proposition: having a cycle} the limit on the left equals $1$, therefore the limit on the right must equal $1$.
\end{proof}

\begin{proof}[Proof of Theorem \ref{Theorem: Main theorem normality 2}]
From Duality criterion (Theorem \ref{Duality criterion}) we have
$$\{I_c(\mathcal{G})\text{ is not normal},\beta_0(\mathcal{G})\leq 2\}\subset \{\bar{\mathcal{G}}\text{ has a Hochster configuration}\},$$
therefore
\begin{multline*}
\PP(\{I_c(\mathcal{G})\text{ is not normal},\beta_0(\mathcal{G})\leq 2\}) \\
\leq \PP(\{\bar{\mathcal{G}}\text{ has a Hochster configuration}\})
\leq \PP(\bar{\mathcal{G}}\text{ has a cycle})
\end{multline*}
So it is enough to prove $\PP(\{\bar{\mathcal{G}}\text{ has a cycle}\})\rightarrow 0$. To prove it, observe that the probability of having an edge in $\bar{\mathcal{G}}$ is the same as not having such an edge in $\mathcal{G}$ which is $q$ and viceversa. Therefore we can replicate the proof of Proposition \ref{Proposition: having a cycle} for $\bar{\mathcal{G}}$ just switching $p$ and $q$ to conclude that
$$\lim_{n\rightarrow\infty}\PP(\{\bar{\mathcal{G}}\text{ has a cycle}\})=0,$$
provided $q=o(\frac{1}{n})$.
\end{proof}

\begin{proof}[Proof of Theorem \ref{Theorem: Main resul Krull}] 
Let us remind that dim$(S/I(\mathcal{G}))$ coindices with the independence number of the graph, this means that the following events are the same
$$\{dim(S/I(\mathcal{G}))\geq t\}=\{\mathcal{G} \text{ contains }E_t\text{ as a subgraph}\},$$
where $E_t$ is defined as in notation \ref{Notation: The graphs T and Ct}. Thus
\begin{equation}\label{Equation: Aux3}
\PP(dim(S/I(\mathcal{G}))\geq t)=\PP(Y_{E_t}(n,p)>0).
\end{equation}
So we are reduced to finding upper and lower bounds for $\PP(Y_{E_t}(n,p)>0)$. An upper bound can be found using Markov's inequality and Equation (\ref{Equation: Expectation of Ct}), namely
\begin{eqnarray*}
\PP(Y_{E_t}(n,p)>0) \leq \mathbb{E}(Y_{E_t}(n,p)) = \binom{n}{t}(1-p)^{\binom{t}{2}}.
\end{eqnarray*}
For fixed $t$, $\binom{n}{t}\leq  n^t/t! \leq n^t$ therefore 
$$\PP(Y_{E_t}(n,p)>0) \leq n^t (1-p)^{\binom{t}{2}} = [n^{2/(t-1)}(1-p)]^{\frac{t(t-1)}{2}}\rightarrow 0, \text{ as }n\rightarrow\infty$$
as long as $n^{2/(t-1)}(1-p)$ converges to $0$ as $n\rightarrow\infty$. Combining this fact with Equation (\ref{Equation: Aux3}) yields
$$\PP(dim(S/I(\mathcal{G})\geq t)\rightarrow 0, \text{ as }n\rightarrow\infty,$$
provided $q=o(\frac{1}{n^{2/(t-1)}})$. Now we provide a lower bound using Corollary \ref{Corollary: Lower bound of Ct}
\begin{equation}\label{Equation: Aux4}
\PP(Y_{E_t}(n,p)>0)\geq 1-\frac{n^{2t}}{\binom{n}{t}^2}\sum_{j=2}^t \frac{1}{n^j(1-p)^{\binom{j}{2}}}.
\end{equation}
Observe that for $2\leq j\leq t$, 
$$n^{j}q^{\binom{j}{2}}=[n^{2/(j-1)}q]^{\binom{j}{2}} \geq [n^{2/(t-1)}q]^{\binom{j}{2}}\rightarrow \infty, \text{ as }n\rightarrow\infty,$$
provided $n^{2/(t-1)}q$ diverges as $n\rightarrow\infty$. This means that each of the terms in the sum of inequality (\ref{Equation: Aux4}) approaches to $0$ as $n\rightarrow\infty$ provided $q=\omega(\frac{1}{n^{2/(t-1)}})$. Further the term $\frac{n^{2t}}{\binom{n}{t}^2}$ converges to a constant (depending on $t)$, hence from (\ref{Equation: Aux4}) and (\ref{Equation: Aux3}) we conclude that
$$\PP(dim(S/I(\mathcal{G}))\geq t)\rightarrow 1,\text{ as }n\rightarrow\infty,$$
provided $q=\omega(\frac{1}{n^{2/(t-1)}})$. This concludes the proof.
\end{proof}

\begin{proof}[Proof of Corollary \ref{Corollary: Main theorem Krull}]
Observe that
$$\PP(dim(S/I(\mathcal{G}))=t)=\PP(dim(S/I(\mathcal{G}))\geq t)-\PP(dim(S/I(\mathcal{G}))\geq t+1).$$
If $q=o(\frac{1}{n^2/t})$ and $q=\omega(\frac{1}{n^{2/(t-1)}})$ then from Theorem \ref{Theorem: Main resul Krull} the first term converges to $1$ while the second term converges to $0$ giving the desired result.
\end{proof}

\begin{proof}[Proof of Corollary \ref{Corollary: Main result regularity}]
From \cite[Proposition~3.2]{v-number} it follows that 
$$reg(S/I(\mathcal{G}))\leq dim(S/I(\mathcal{G})),$$
thus
$$1\geq \PP(reg(S/I(\mathcal{G}))\leq t-1) \geq \PP(dim(S/I(\mathcal{G}))\leq t-1),$$
from where we conclude from Theorem \ref{Theorem: Main resul Krull} as the right hand side converges to $1$ provided $q=o(\frac{1}{n^{2/(t-1)}})$.
\end{proof}

\begin{proof}[Proof of Corollary \ref{Corollary: Main result v number}]
From \cite[Corollary~3.6]{v-number} it follows that 
$$v(I(\mathcal{G}))\leq \beta_0(\mathcal{G})=dim(S/I(\mathcal{G})),$$
thus
$$1\geq \PP(v(I(\mathcal{G}))\leq t-1) \geq \PP(dim(S/I(\mathcal{G}))\leq t-1),$$
from where we conclude from Theorem \ref{Theorem: Main resul Krull} as the right hand side converges to $1$ provided $q=o(\frac{1}{n^{2/(t-1)}})$.
\end{proof}

\section*{Acknowledgments}

We would like to thank Professor Fiona Skerman for their online notes on random graphs from where we inspired some of our proofs.


\begin{thebibliography}{XX}

\bibitem{AS}
N.~Alon, J.~H.~Spencer, The Probabilistic Method, Wiley, 4th ed., 2016.

\bibitem{JHB}
J.~C.~Baez, Groupoid cardinality and random permutations. \textit{Theory Appl. Categ.}, \textbf{44}, (2025), Paper No. 14, 410–419.

\bibitem{GB24}
Z.~Bao, D.~Munoz George, Ultra high order cumulants and quantitative CLT for polynomials in random matrices, \textit{Probab. Theory Related Fields}, 1-51.

\bibitem{BPEY}
C.~Booms-Peot, D.~Erman, J.~Yang, Characteristic dependence of syzygies of random monomial ideals. \textit{SIAM J. Discrete Math.}, \textbf{36}, (2022), no. 1, 682–701.

\bibitem{CCMP}
Y.~Chang, Q.~Chen, Q.~Meng, X.~Peng, Strong Law of Large Numbers for a Function of the Local Time of a Transient Random Walk on a Group. \textit{J. Theoret. Probab.}, \textbf{39}, (2026), no. 1, 7.

\bibitem{CWY}
G.~Cho, G.~Wei, G.~Yang, Probabilistic method to fundamental gap problems on the sphere. \textit{Trans. Amer. Math. Soc.}, \textbf{378}, (2025), no. 1, 317–337.

\bibitem{CES23}
G.~Cipolloni, L.~Erd\H{o}s, D.~Schr\"{o}der, Functional central limit theorems for Wigner matrices. {\it Ann. Appl. Probab.}, {\bf 33}(1), (2023), 447-89.

\bibitem{min-dis-generalized} S. M. ~Cooper, A. ~Seceleanu, S. ~O. ~Toh\v{a}neanu, M. ~Vaz ~Pinto and R. ~H. ~Villarreal, 
Generalized minimum distance functions and algebraic invariants of
Geramita ideals, \textit{Adv. in Appl. Math.} {\bf 112}, (2020), 101940.

\bibitem{ranmon}
J.~A.~De Loera, S.~Petrović, L.~Silverstein, D.~Stasi, D.~Wilburne, Random monomial ideals. \textit{J. Algebra} \textbf{519}, (2019), 440–473.

\bibitem{icdual} L. ~A. ~Dupont, H. ~Muñoz ~George and R.~H. ~Villarreal, Normality criteria for monomial ideals, \textit{Results Math.}, \textbf{78}, (2023), no. 1, paper no. 34, 31 pp.

\bibitem{Eisen} D. ~Eisenbud, Commutative Algebra with a view toward Algebraic Geometry, \textit{Graduate Texts in  Mathematics}, {\bf 150}, (1995), Springer-Verlag, New York.

\bibitem{ER}
P.~Erdos, A.~Rényi, On random graphs, \textit{Publ. Math. Debrecen}, \textbf{6}, (1959), pp. 290–297.

\bibitem{G59}
E.~Gilbert, Random graphs, \textit{Ann. Math. Statist.}, \textbf{30}, (1959), pp. 1141–1144.

\bibitem{alexdual} I. ~Gitler, E. ~Reyes and R. ~H. ~Villarreal, Blowup algebras of ideals of vertex covers of bipartite graphs, Contemp. 
Math., {\bf 376}, (2005), 273--279. 

\bibitem{graphs} I.~Gitler and R.~H.~Villarreal, {\it Graphs, Rings and
Polyhedra\/}, Aportaciones Mat. Textos, {\bf 35}, Soc. Mat. Mexicana,
M\'exico, 2011.  

\bibitem{G}
M.~Gromov, Random walk in random groups, \textit{Geom. Funct. Anal.}, \textbf{13}, (2003), 73–146

\bibitem{Herzog} J.~Herzog, T.~Hibi, Monomial Ideals, \textit{Graduate Texts in Mathematics}, \textbf{Vol. 260},
Springer (2011). 

\bibitem{v-number} D. ~Jaramillo and R. ~H. ~Villarreal, The
v-number of edge ideals, \textit{J. Combin. Theory Ser. A}, {\bf 177}, (2021), Paper 105310, 35 pp. 

\bibitem{KKP}
A.~M.~Khorunzhy, B.~A.~Khoruzhenko, and L.~A.~Pastur, Asymptotic properties of large random matrices with independent entries. \textit{J. Math. Phys.}, \textbf{37}, (1996), 5033-5060.

\bibitem{LSX}
Y.~Li, K.~Schnelli, Y.~Xu, Central limit theorem for mesoscopic eigenvalue statistics of deformed Wigner matrices and sample covariance matrices. \textit{Ann. Inst. H. Poincar\'{e} Probab. Statist.} , \textbf{57}(1), (2021), 506-546

\bibitem{LWW}
S.~Liu, N.~Wang, N.~Wang, Multifractal of random permutation set. \textit{Comm. Statist. Theory Methods}, \textbf{54}, (2025), no. 20, 6562–6591.

\bibitem{LP1}
A.~Lytova, L.~Pastur, Central limit theorem for linear eigenvalue statistics of random matrices with independent
entries. \textit{Ann. Probab.}, {\bf 37}, Number 5 (2009), 1778-1840.

\bibitem{PSW}
S.~Petrović, D.~Stasi, D.~Wilburne, Random monomial ideals: a Macaulay2 package. \textit{J. Softw. Algebra Geom.}, textbf{9}, (2019), no. 1, 65–70.

\bibitem{Sh11} M.~Shcherbina, Central limit theorem for linear eigenvalue statistics of the Wigner and sample covariance random matrices. {\it J. Math. Phys. Anal. Geom.}, {\bf 7}(2), (2011), 176-192.

\bibitem{SWY}
L.~Silverstein, D.~Wilburne, J.~Yang, Asymptotic degree of random monomial ideals. \textit{J. Commut. Algebra}, \textbf{15}, (2023), no. 1, 99–114.

\bibitem{OITG}
A. ~Simis, W. ~Vasconcelos, R. ~H. ~Villarreal, R.: On the ideal theory of graphs. \textit{J. Algebra} \textbf{167}, (1994) 389–416.

\bibitem{cmg}{R. ~H.~Villarreal, Cohen--{M}acaulay graphs, \textit{Manuscripta Math.}, {\bf 66}, (1990), 277--293.}

\bibitem{monalg-rev} R. ~H. ~Villarreal, {\it Monomial Algebras, Second Edition\/}, Monographs and Research Notes in Mathematics, Chapman and Hall/CRC, Boca Raton, FL, 2015.

\bibitem{RAPC} R. ~H. ~Villarreal, Rees algebras and polyhedral cones of ideals of vertex covers of
perfect graphs., \textit{J. Algebraic Combin.}, \textbf{27}, (2008), 293–305.

\bibitem{V} D.~Voiculescu, Limit laws for random matrices and free products, \textit{Invent. Math.}, \textbf{104}, (1991), 201-220.

\bibitem{VDN92}
D.~V.~Voiculescu, K.~J.~Dykema, and A.~Nica. Free random variables. American Mathematical Society, Providence, RI, 1992.

\bibitem{HMZ}
H.~M.~Zhylinskyi, Random braids and random walks on finite groups. \textit{Theory Stoch. Process.} \textbf{29}, (2025), no. 1, 129–138.


\end{thebibliography}
\end{document}